\documentclass[a4paper,12pt]{amsart}
\usepackage[margin=3truecm,top=4truecm,bottom=4truecm]{geometry}
\usepackage{amsthm,amsmath,amssymb,setspace,mleftright,constants,cite}
\allowdisplaybreaks %%% Allow display breaks
\usepackage[pagewise]{lineno}
\usepackage{etoolbox} %% <- for \pretocmd and \apptocmd
\makeatletter %% <- make @ usable in macro names
\newcommand*\linenomathpatch{\@ifstar{\linenomathpatch@AMS}{\linenomathpatch@}}
\newcommand*\linenomathpatch@[1]{
	\expandafter\pretocmd\csname #1\endcsname {\linenomathWithnumbers}{}{}
	\expandafter\pretocmd\csname #1*\endcsname{\linenomathWithnumbers}{}{}
	\expandafter\apptocmd\csname end#1\endcsname {\endlinenomath}{}{}
	\expandafter\apptocmd\csname end#1*\endcsname{\endlinenomath}{}{}
}
\newcommand*\linenomathpatch@AMS[1]{
	\expandafter\pretocmd\csname #1\endcsname {\linenomathWithnumbersAMS}{}{}
	\expandafter\pretocmd\csname #1*\endcsname{\linenomathWithnumbersAMS}{}{}
	\expandafter\apptocmd\csname end#1\endcsname {\endlinenomath}{}{}
	\expandafter\apptocmd\csname end#1*\endcsname{\endlinenomath}{}{}
}
\let\linenomathWithnumbersAMS\linenomathWithnumbers
\patchcmd\linenomathWithnumbersAMS{\advance\postdisplaypenalty\linenopenalty}{}{}{}
\makeatother %% revert @

\linenomathpatch{equation}
\linenomathpatch*{gather}
\linenomathpatch*{multline}
\linenomathpatch*{align}
\linenomathpatch*{alignat}
\linenomathpatch*{flalign}
\makeatletter

\@addtoreset{equation}{section}
\makeatother
\newconstantfamily{c}{symbol=c}
\newconstantfamily{E}{symbol=\mathcal{E}}
\theoremstyle{plain} % plain, definition, remark
\newtheorem{thm}{Theorem}[section]
\newtheorem{prop}[thm]{Proposition}
\newtheorem{cor}[thm]{Corollary}
\newtheorem{lem}[thm]{Lemma}

\theoremstyle{definition}

\theoremstyle{remark}
\newtheorem{rem}[thm]{Remark}

\newcommand{\N}{\mathbb{N}}
\newcommand{\R}{\mathbb{R}}
\newcommand{\Z}{\mathbb{Z}}

\renewcommand{\P}{\mathbb{P}}
\newcommand{\E}{\mathbb{E}}

\newcommand{\1}[1]{\mathbf{1}_{#1}}

\renewcommand{\tilde}{\widetilde}

\renewcommand{\bar}{\overline}

\newcommand{\hyphen}{\textrm{-}}
\newcommand{\as}{\textrm{a.s.}}

\renewcommand{\S}{\mathbb{S}}

\begin{document}
% Info ====================================================
\title[Differences between Lyapunov exponents]{Differences between Lyapunov exponents for the simple random walk in Bernoulli potentials}
\author[N.~KUBOTA]{Naoki KUBOTA}
\address[N. Kubota]{College of Science and Technology, Nihon University, Chiba 274-8501, Japan.}
\email{kubota.naoki08@nihon-u.ac.jp}
\thanks{The author was supported by JSPS KAKENHI Grant Number JP20K14332.}
%\date{}

\begin{abstract}
We consider the simple random walk on the $d$-dimensional lattice $\Z^d$ ($d \geq 1$), traveling in potentials which are Bernoulli distributed.
The so-called Lyapunov exponent describes the cost of traveling for the simple random walk in the potential, and it is known that the Lyapunov exponent is strictly monotone in the parameter of the Bernoulli distribution.
Hence, the aim of this paper is to investigate the effect of the potential on the Lyapunov exponent more precisely, and we derive some Lipschitz-type estimates for the difference between the Lyapunov exponents.
\end{abstract}
% =========================================================

\maketitle
%\tableofcontents
%\linenumbers

\section{Introduction}
In this paper, we consider the simple random walk in i.i.d.~nonnegative potentials on the $d$-dimensional lattice $\Z^d$ ($d \geq 1$).
A central object of its study is the so-called Lyapunov exponent, which measures the cost paid by the simple random walk for traveling in a landscape of potentials.
\cite{Kub20_arXiv} proved that the Lyapunov exponent strictly increases in the law of potential with some order.
Considering this result, the natural question arises as to how much the change in the law of potential affects the Lyapunov exponent.
Hence, the goal of this paper is to investigate this problem in the case where the potential is Bernoulli distributed.

\subsection{The model}
Let $\S=(S_k)_{k=0}^\infty$ be the simple random walk on $\Z^d$.
For $x \in \Z^d$, we write $P^x$ and $E^x$ for the law of the simple random walk starting at $x$ and its associated expectation, respectively.
Independently of $\S$, let $\omega=(\omega(x))_{x \in \Z^d}$ be a family of i.i.d.~random variables taking values in $[0,\infty)$, and $\omega$ is called the \emph{potential}.
Denote by $\P$ and $\E$ the law of the potential $\omega$ and its associated expectation, respectively.
%Our interest is to study the behavior of the simple random walk $(S_k)_{k=0}^\infty$ on $\Z^d$ in the potential $\omega$.
%For $d \geq 1$, write $\Z^d$ for the $d$-dimensional cubic lattice.
%Let $\omega=(\omega(x))_{x \in \Z^d}$ be a family of i.i.d.~random variables, whose law is $\P_p$ (with the associated expectation $\E_p$).
%We call this family $\omega$ the \emph{Bernoulli potential} with parameter $p$.
%Independently of $\omega$, let $(S_k)_{k=0}^\infty$ be the simple random walk on $\Z^d$.
%For $x \in \Z^d$, write $P^x$ and $E^x$ for the law of the simple random walk starting at $x$ and the associated expectation, respectively.
%Given $p \in [0,1]$, let $\omega=(\omega(x))_{x \in \Z^d}$ be a family of i.i.d.~random variables, whose law is $\P_p$ (with the associated expectation $\E_p$), satisfying that
%\begin{align*}
%	\P_p(\omega(x)=0)=1-\P_p(\omega(x)=1)=p,\qquad x \in \Z^d.
%\end{align*}
%We call this family $\omega$ the \emph{Bernoulli potential} with parameter $p$.
%Independently of $\omega$, let $(S_k)_{k=0}^\infty$ be the simple random walk on $\Z^d$.
%For $x \in \Z^d$, write $P^x$ and $E^x$ for the law of the simple random walk starting at $x$ and the associated expectation, respectively.

Let us now introduce the Lyapunov exponent, which is the main object of study in the present article.
For any $y \in \Z^d$, $H(y)$ stands for the hitting time of the simple random walk to $y$, i.e.,
\begin{align*}
	H(y):=\inf\{ k \geq 0:S_k=y \}.
\end{align*}
Moreover, define for $x,y \in \Z^d$,
\begin{align*}
	e(x,y,\omega)
	:= E^x\Biggl[ \exp\Biggl\{ -\sum_{k=0}^{H(y)-1}\omega(S_k) \Biggr\} \1{\{ H(y)<\infty \}} \Biggr],
\end{align*}
with the convention that $e(x,y,\omega):=1$ if $x=y$.
Then, the following quantities $a(x,y,\omega)$ and $b(x,y)$ are called the \emph{quenched} and \emph{annealed travel costs} from $x$ to $y$, respectively:
\begin{align*}
	a(x,y,\omega):=-\log e(x,y,\omega)
\end{align*}
and
\begin{align*}
	b(x,y):=-\log\E[e(x,y,\omega)].
\end{align*}
The asymptotic behavior of each travel cost induces a norm on $\R^d$, and it is the Lyapunov exponent.
More precisely, Flury~\cite[Theorem~A]{Flu07}, Mourrat~\cite[Theorem~1.1]{Mou12} and Zerner~\cite[Proposition~4]{Zer98a} obtained the following result:
There exist (nonrandom) norms $\alpha(\cdot)$ and $\beta(\cdot)$ on $\R^d$ such that for all $x \in \Z^d$,
\begin{align*}
	\lim_{n \to \infty} \frac{1}{n}a(0,nx,\omega)
	&= \lim_{n \to \infty} \frac{1}{n}\E[a(0,nx,\omega)]\\
	&= \inf_{n \in \N} \frac{1}{n}\E[a(0,nx,\omega)]
		= \alpha(x) \qquad \text{$\P$-a.s.~and in $L^1(\P)$},
\end{align*}
and
\begin{align*}
	\lim_{n \to \infty} \frac{1}{n}b(0,nx)
	= \inf_{n \in \N} \frac{1}{n} b(0,nx)
	= \beta(x).
\end{align*}
We call $\alpha(\cdot)$ and $\beta(\cdot)$ the \emph{quenched} and \emph{annealed Lyapunov exponents}, respectively.

\begin{rem}
By definition, the annealed travel cost $b(x,y)$, the Lyapunov exponents $\alpha(\cdot)$ and $\beta(\cdot)$ depend on the distribution function of $\omega(0)$, say $\phi$.
From now on, we basically put a subscript $\phi$ on the above notations: $b(x,y)=b_\phi(x,y)$, $\alpha(\cdot)=\alpha_\phi(\cdot)$ and $\beta(\cdot)=\beta_\phi(\cdot)$.
\end{rem}

\subsection{Main results}
First of all, we state the motivation for the present work.
Let $F$ and $G$ be distribution functions on $[0,\infty)$ such that $F \leq G$, $F(0)<G(0)$ and $\int_0^\infty t\,dF(t)$ is finite.
Then, \cite[Theorems~{1.4} and {1.5}]{Kub20_arXiv} gives the following strict comparisons for the quenched and annealed Lyapunov exponents:
There exist constants $0<C,C'<\infty$ (which may depend on $d$, $F$ and $G$) such that for all $x \in \R^d \setminus \{0\}$,
\begin{align}\label{eq:diff_ab}
	\alpha_F(x)-\alpha_G(x) \geq C\|x\|_1 \text{ and } \beta_F(x)-\beta_G(x) \geq C'\|x\|_1,
\end{align}
where $\|\cdot\|_1$ is the $\ell^1$-norm on $\R^d$.
The above inequalities do not provide any information on how much the differences in \eqref{eq:diff_ab} are affected by $F$ and $G$.
Hence, the goal of this paper is to estimate the differences in \eqref{eq:diff_ab} more precisely by focusing on Bernoulli distributed potentials.
%It may be difficult to solve this problem for general potentials, and the goal of this paper is to estimate the differences in \eqref{eq:diff_ab} more precisely for Bernoulli distributed potentials.

Let $0 \leq r \leq 1$ and $\P_r$ is the law of the potential $\omega$ defined by
\begin{align}\label{eq:potential}
	\P_r(\omega(x)=0)=1-\P_r(\omega(x)=1)=r,\qquad x \in \Z^d.
\end{align}
In this situation, we call $\omega$ the \emph{Bernoulli potential} with parameter $r$.
Moreover, write $F_r$ and $\E_r$ for the distribution function and the expectation with respect to $\P_r$, respectively.
To shorten notation, set
\begin{align*}
	b_r(x,y):=b_{F_r}(x,y),\quad \alpha_r(\cdot):=\alpha_{F_r}(\cdot),\quad \beta_r(\cdot):=\beta_{F_r}(\cdot).
\end{align*}
Then, the following two theorems are our main results, which estimate differences between quenched and annealed Lyapunov exponents.

\begin{thm}\label{thm:diff_qlyap}
We have the following lower and upper bounds for differences between quenched Lyapunov exponents:
\begin{enumerate}
	\item\label{item:diff_ql1}
		We have for all $0<p<q<1$,
		\begin{align*}
			\inf_{x \in \R^d\setminus\{0\}}\frac{\alpha_p(x)-\alpha_q(x)}{\|x\|_1}
			\geq (1-e^{-1})(q-p).
		\end{align*}
	\item\label{item:diff_ql2}
		Let $0<q_0<1$.
		Then, there exists a constant $0<\Cl{qlyap}<\infty$ (which may depend on $d$ and $q_0$) such that for all $0<p<q \leq q_0$,
		\begin{align*}
			\sup_{x \in \R^d\setminus\{0\}} \frac{\alpha_p(x)-\alpha_q(x)}{\|x\|_1} \leq \Cr{qlyap}(q-p).
		\end{align*}
\end{enumerate}
\end{thm}

\begin{thm}\label{thm:diff_alyap}
The following results hold for the annealed Lyapunov exponent:
\begin{enumerate}
	\item\label{item:diff_al1}
		We obtain the statement of part~\eqref{item:diff_ql1} in Theorem~\ref{thm:diff_qlyap} with $\alpha_p(\cdot)$ and $\alpha_q(\cdot)$ replaced by $\beta_p(\cdot)$ and $\beta_q(\cdot)$, respectively.
	\item\label{item:diff_al2}
		Let $0<q_0<1$.
		Then, there exists a constant $0<\Cl{alyap}<\infty$ (which may depend on $d$ and $q_0$) such that for all $0<p<q \leq q_0$,
		\begin{align}\label{eq:diff_aupper}
			\sup_{x \in \R^d \setminus \{0\}}
			\frac{\beta_p(x)-\beta_q(x)}{\|x\|_1} \leq \Cr{alyap}(\log q-\log p).
		\end{align}
		In particular, if $d \geq 3$, then the right side of \eqref{eq:diff_aupper} can be replaced by $\Cr{alyap}(q-p)$.
\end{enumerate}
\end{thm}

\begin{rem}\label{rem:annealed}
Le~\cite{Le17} obtained the (ordinary) continuity of the quenched and annealed Lyapunov exponents in the law of potential.
From this point of view, Theorems~\ref{thm:diff_qlyap} and \ref{thm:diff_alyap} provide stronger results than the (ordinary) continuity in the case where the potential is Bernoulli distributed.
Moreover, \cite{Le17} sometimes requires that the potential is bounded away from zero in the low dimensional case ($d=1,2$).
Since the Bernoulli potential is not bounded away from zero, Theorems~\ref{thm:diff_qlyap} and \ref{thm:diff_alyap} also exhibit the continuity of the quenched and annealed Lyapunov exponents which \cite{Le17} does not deal with.
\end{rem}

\begin{rem}
At first, we believed that the right side of \eqref{eq:diff_aupper} could easily be replaced by $\Cr{alyap}(q-p)$ in $d=1,2$ by using the same argument taken in the proof of part~\eqref{item:diff_ql2} in Theorem~\ref{thm:diff_qlyap}.
Unfortunately, that argument does not work well, and this may suggest that in $d=1,2$, the upper and lower bounds stated in Theorem~\ref{thm:diff_alyap} are not correct.
However, if the range of $p$ and $q$ is restricted to a closed interval $[p_0,q_0] \subset (0,1)$, then the mean value theorem implies that there exists a constant $0<\Cr{alyap}'<\infty$ (which may depend on $d$, $p_0$ and $q_0$) such that for all $p_0 \leq p<q \leq q_0$, the right side of \eqref{eq:diff_aupper} is bounded from above by $\Cr{alyap}'(q-p)$ in $d=1,2$.
Consequently, even if $d=1,2$, in the case where $p_0 \leq p<q \leq q_0$, the quantity $q-p$ dominates the upper and lower bounds between the annealed Lyapunov exponents.
%Although we cannot answer this question in the present article, in view of the proof of Theorem~\ref{thm:diff_alyap}, we probably need further investigation for random lattice animals and multiple points for the simple random walk in the Bernoulli potential.
\end{rem}

Let us finally comment on earlier works related to the above results.
Recently, the coincidence of the quenched and annealed Lyapunov exponents has been studied positively:
Flury~\cite{Flu08} and Zygouras~\cite{Zyg09} proved that the quenched and annealed Lyapunov exponents coincide in $d \geq 4$ and the low disorder regime.
Furthermore, Wang~\cite{Wan01,Wan02} and Kosygina et al.~\cite{KosMouZer11} studied the asymptotic behavior of the quenched and annealed Lyapunov exponents as the potential tends to zero.

On the other hand, this paper treats the comparison between (quenched/annealed) Lyapunov exponents for different laws of the potential.
As mentioned at the beginning of this subsection, the strict comparison for the Lyapunov exponents is obtained in \cite{Kub20_arXiv}, and the goal of this article is to estimate more precisely differences between quenched and annealed Lyapunov exponents when the law of potential is restricted to the Bernoulli distribution.
Even if we focus on the Bernoulli setting, it is not easy to analyze the Lyapunov exponent precisely.
Actually, there are not so many researches related to our topic:
In \cite{CerDem21_arXiv,Dem21} and \cite{Dem20}, the Lipschitz continuity is obtained for the so-called time constant of the first passage percolation on $\Z^d$ and isoperimetric constant of the supercritical percolation cluster (which are counterparts of the Lyapunov exponent), respectively.
The aim of \cite{WuFen09} is to derive a (nontrivial) lower bound for the difference between the time constants.
This is a similar topic to part~\eqref{item:diff_ql1} in Theorem~\ref{thm:diff_qlyap}, and the key idea for its proof is actually inspired by \cite{WuFen09}.
However, the counterpart of the travel cost used there takes only nonnegative integer values, and this is not always true for the travel cost.
Hence, we have to take a different approach from \cite{WuFen09}.
Fortunately, this modified approach is also useful for the proof of part~\eqref{item:diff_ql2} in Theorem~\ref{thm:diff_qlyap}.

\subsection{Organization of the paper}
Let us describe how the present article is organized.
In Section~\ref{sect:qlower_lyap}, we prove part~\eqref{item:diff_ql1} of Theorem~\ref{thm:diff_qlyap}, which gives a lower bound for the difference between quenched Lyapunov exponents.
Note that the quenched Lyapunov exponent is described by the expectation of the quenched travel cost.
Hence, our main task is to estimate differences between expectations of quenched travel costs from below.
For this purpose, we use Russo's formula for the independent Bernoulli site percolation on $\Z^d$.
Russo's formula enables us to differentiate the expectation for the quenched travel cost with respect to parameter $r$ (see Proposition~\ref{prop:Russo} below).
Then, a standard calculation shows that the obtained derivative is bounded from above uniformly in $r$.
This implies our desired conclusion since the expectation of the quenched travel cost is decreasing in $r$.

Section~\ref{sect:qupper} is devoted to the proof of part~\eqref{item:diff_ql2} in Theorem~\ref{thm:diff_qlyap}, which gives an upper bound for the difference between quenched Lyapunov exponents.
Considering the proof of part~\eqref{item:diff_ql1} in Theorem~\ref{thm:diff_qlyap}, it suffices to estimate the derivative of the expectation for the quenched travel cost from below uniformly in parameter $r$.
To this end, we divide the proof into two cases: Sections~\ref{ssect:high} and \ref{ssect:low} treat the high dimensional case ($d \geq 3$) and the low dimensional case ($d=1,2$), respectively.
In the high dimensional case, the transient of the simple random walk strongly effectively works, and the derivative of the expectation for the quenched travel cost is bounded from below uniformly in $r$.
On the other hand, the simple random walk is recurrent in the low dimensional case.
This is the reason why the proof is divided into two cases, and we need some more analysis to estimate the derivative from below uniformly in $r$.
Actually, in the low dimensional case, the derivative is related to the range of the simple random walk equipped with random connected subsets of $\Z^d$.
Hence, our efforts are focused on estimating this range under the path measure defined as in \eqref{eq:qpathmeas}.
To this end, we rely on the analysis of lattice animals developed by Fontes--Newman~\cite{FonNew93} and Scholler~\cite{Sch14}.
Roughly speaking, this analysis guarantees that even if $d=1,2$, the potential makes the simple random walk transient under the path measure.
Therefore, similar to the high dimensional case, we can also estimate the derivative from below uniformly in $r$.

The aim of Section~\ref{sect:abound} is to prove Theorem~\ref{thm:diff_alyap}, which gives both lower and upper bounds for the difference between annealed Lyapunov exponents.
Similar to the above, our task here is to estimate the derivative of the annealed travel cost from above and below.
By differentiating under the integral sign, the derivative can be written as the expectation of some random variable with respect to the path measure defined as in \eqref{eq:apathmeas} (see Lemma~\ref{lem:derivative_b} below).
This random variable consists of the function $f(t):=(1-t)/\{ r+t(1-r) \}$, $0 \leq t \leq 1$, and the property of $f(t)$ implies that the derivative is bounded from above uniformly in $r$ and has a lower bound of order $r^{-1}$.
In particular, in the high dimensional case, the derivative can be bounded from below uniformly in $r$ due to transience of the simple random walk (see Lemma~\ref{lem:l_bound} below).

In Section~\ref{sect:LDP}, we comment on the quenched and annealed large deviation principles for the simple random walk in the Bernoulli potential.
It is known that their rate functions are described by the Lyapunov exponents for the potentials $\omega+\lambda=(\omega(x)+\lambda)_{x \in \Z^d}$, $\lambda \geq 0$.
Furthermore, Theorems~\ref{thm:diff_qlyap} and \ref{thm:diff_alyap} are also established for these Lyapunov exponents since the same arguments as in Sections~\ref{sect:qlower_lyap}--\ref{sect:abound} work, and we can estimate differences between quenched and annealed rate functions.

We close this section with some general notation.
Write $\|\cdot\|_1$ and $\|\cdot\|_\infty$ for the $\ell^1$ and $\ell^\infty$-norms on $\R^d$.
Throughout the paper, $c$, $c'$ and $C_i$, $i=1,2,\dots$, denote some constants with $0<c,c',C_i<\infty$.

\section{Lower bound for the quenched Lyapunov exponent}\label{sect:qlower_lyap}
The aim of this section is to prove part~\eqref{item:diff_ql1} of Theorem~\ref{thm:diff_qlyap}.
The key tool here is Russo's formula for the independent Bernoulli site percolation on $\Z^d$.
Roughly speaking, by applying Russo's formula to some events with respect to the quenched travel cost, we derive a lower bound for the differences between quenched Lyapunov exponents.
However, Russo's formula is not directly applicable because the quenched travel cost depends on the states of infinitely many sites.
To overcome this problem, we first introduce a modification of the quenched travel cost depending only on the states of finitely many sites.
%and state a result obtained by applying Russo's formula to it (This result is also useful to derive an upper bound for the difference between the quenched Lyapunov exponents).
%The proof of part~\eqref{item:diff_qlyap1} in Theorem~\ref{thm:diff_qlyap} is given in Section~\ref{ssect:diff_qlyap1}.

Let $V$ be a subset of $\R^d$ and denote by $T(V)$ the exit time of the simple random walk from $V$, i.e.,
\begin{align}\label{eq:exit}
	T(V):=\inf\{ k \geq 0:S_k \not\in V \}.
\end{align}
Then, for any $x,y \in \Z^d$ and $N \in \N$, we define the quenched travel cost $a_N(x,y,\omega)$ from $x$ to $y$ restricted to the simple random walk before exiting $[-N,N]^d$ as follows:
\begin{align*}
	a_N(x,y,\omega):=-\log e_N(x,y,\omega),
\end{align*}
where
\begin{align*}
	e_N(x,y,\omega)
	:= E^x\Biggl[ \exp\Biggl\{ -\sum_{k=0}^{H(y)-1}\omega(S_k) \Biggr\} \1{\{ H(y)<T([-N,N]^d) \}} \Biggr]
\end{align*}
if $x \not= y$, and $e_N(x,y,\omega):=1$ otherwise.
%Let $N \in \N$ be large enough and set $V_N:=[-N,N]^d$.
%To simplify notation, we write for $\lambda \geq 0$ and $x,y \in \Z^d$,
%\begin{align*}
%	e_N(x,y,\omega):=e_\lambda^{V_N}(x,y,\omega)
%\end{align*}
%and
%\begin{align*}
%	H(x,y,\omega):=a_\lambda^{V_N}(x,y,\omega).
%\end{align*}
Note that $a_N(x,y,\omega)$ depends only on the states of finitely many sites, and some convergence theorems show that for each $0 \leq r \leq 1$,
\begin{align}\label{eq:c}
	\lim_{N \to \infty} \E_r[a_N(x,y,\omega)]
	= \E_r[a(x,y,\omega)].
\end{align}

We need some preparation before applying Russo's formula to the restricted quenched travel cost.
For any $y \in \Z^d$, the path measure $\tilde{P}_{N,\omega}^{0,y}$ is defined by
\begin{align}\label{eq:qpathmeas}
	\frac{d\tilde{P}_{N,\omega}^{0,y}}{dP^0}
	= e_N(0,y,\omega)^{-1} \exp\Biggl\{ -\sum_{k=0}^{H(y)-1}\omega(S_k) \Biggr\}
	\1{\{ H(y)<T([-N,N]^d) \}},
\end{align}
and denote by $\tilde{E}_{N,\omega}^{0,y}$ the expectation with respect to $\tilde{P}_{N,\omega}^{0,y}$.
Next, for any $z \in \Z^d$ and $m \in \N$, let $\ell_z(m)$ be the number of visits to $z$ by the simple random walk up to time $m-1$, i.e.,
\begin{align*}
	\ell_z(m):=\#\{ 0 \leq k<m:S_k=z \}.
\end{align*}

The following proposition is a consequence of Russo's formula for the restricted quenched travel cost, which is also useful to derive an upper bound for differences between quenched Lyapunov exponents.

\begin{lem}\label{prop:Russo}
We have for all $0<r<1$, $y \in \Z^d \setminus \{0\}$ and $N \in \N$ with $N \geq \|y\|_\infty$,
\begin{align*}
	-\frac{d}{dr}\E_r[a_N(0,y,\omega)]
	&= \sum_{z \in \Z^d \cap [-N,N]^d} \Bigl\{ \E_r \Bigl[ \1{\{ \omega(z)=1 \}} \log \tilde{E}_{N,\omega}^{0,y}[e^{\ell_z(H(y))}] \Bigr]\\
	&\hspace{7em}+\E_r\Bigl[ \1{\{ \omega(z)=0 \}}\bigl( -\log \tilde{E}_{N,\omega}^{0,y}[e^{-\ell_z(H(y))}] \bigr) \Bigr] \Bigr\}.
\end{align*}
\end{lem}
\begin{proof}
Fix $y \in \Z^d \setminus \{0\}$ and $N \in \N$ with $N \geq \|y\|_\infty$.
Since $A(t)$ is an increasing event defined in terms of the states of sites in $[-N,N]^d$, Russo's formula (see for instance \cite[Theorem~{2.32}]{Gri99_book}) gives that for any $0<r<1$,
\begin{align}\label{eq:Russo}
	-\frac{d}{dr}\E_r[a_N(0,y,\omega)]
	= \sum_{z \in \Z^d \cap [-N,N]^d} \E_r\Bigl[ a_N(0,y,\omega_z^1)-a_N(0,y,\omega_z^0) \Bigr],
\end{align}
where
\begin{align*}
	\omega_z^1(x):=
	\begin{cases}
		1, & \text{if } x=z,\\
		\omega(x), & \text{if } x \not=z,
	\end{cases}
	\qquad
	\omega_z^0(x):=
	\begin{cases}
		0, & \text{if } x=z,\\
		\omega(x), & \text{if } x \not=z.
	\end{cases}
\end{align*}
Note that if $\omega(z)=1$ holds, then
\begin{align*}
	a_N(0,y,\omega_z^1)-a_N(0,y,\omega_z^0)
	= \log \tilde{E}_{N,\omega}^{0,y}\bigl[ e^{\ell_z(H(y))} \bigr].
\end{align*}
Furthermore, in the case where $\omega(z)=0$,
\begin{align*}
	a_N(0,y,\omega_z^1)-a_N(0,y,\omega_z^0)
	= -\log \tilde{E}_{N,\omega}^{0,y}\bigl[ e^{-\ell_z(H(y))} \bigr].
\end{align*}
This implies that the right side of \eqref{eq:Russo} is equal to
\begin{align*}
	&\sum_{z \in \Z^d \cap [-N,N]^d}\Bigl\{ \E_r \Bigl[ \1{\{ \omega(z)=1 \}} \log \tilde{E}_{N,\omega}^{0,y}[e^{\ell_z(H(y))}] \Bigr]\\
	&\hspace{7em} +\E_r\Bigl[ \1{\{ \omega(z)=0 \}}\bigl( -\log \tilde{E}_{N,\omega}^{0,y}[e^{-\ell_z(H(y))}] \bigr) \Bigr] \Bigr\},
\end{align*}
and the proposition follows.
\end{proof}

We are now in a position to show part~\eqref{item:diff_ql1} in Theorem~\ref{thm:diff_qlyap}.

\begin{proof}[\bf Proof of part~(\ref{item:diff_ql1}) in Theorem~\ref{thm:diff_qlyap}]
Let $y \in \Z^d \setminus \{0\}$ and $N \in \N$ with $N \geq \|y\|_\infty$.
Jensen's inequality and the fact that $-\log(1-t) \geq t$ holds for $0 \leq t<1$ imply that for any $z \in \Z^d \cap [-N,N]^d$,
\begin{align*}
	\log\tilde{E}_{N,\omega}^{0,y}[e^{\ell_z(H(y))}]
	&\geq -\log\tilde{E}_{N,\omega}^{0,y}[e^{-\ell_z(H(y))}]\\
	&\geq -\log\Bigl\{ 1-(1-e^{-1})\tilde{P}_{N,\omega}^{0,y}(\ell_z(H(y)) \geq 1) \Bigr\}\\
	&\geq (1-e^{-1})\tilde{P}_{N,\omega}^{0,y}(\ell_z(H(y)) \geq 1).
\end{align*}
Hence, Lemma~\ref{prop:Russo} yields that for all $0<r<1$,
\begin{align*}
	-\frac{d}{dr}\E_r[a_N(0,y,\omega)]
	&\geq\sum_{z \in \Z^d \cap [-N,N]^d} \E_r \Bigl[ (1-e^{-1})\tilde{P}_{N,\omega}^{0,y}(\ell_z(H(y)) \geq 1) \bigr] \Bigr]\\
	&= (1-e^{-1})\, \E_r\Biggl[\tilde{E}_{N,\omega}^{0,y}\biggl[ \sum_{z \in \Z^d \cap [-N,N]^d} \1{\{ \ell_z(H(y)) \geq 1 \}} \biggr] \Biggr]\\
	&\geq (1-e^{-1})\|y\|_1.
\end{align*}
Here the last inequality follows from the fact that $\tilde{P}_{N,\omega}^{0,y} \hyphen \as$, the simple random walk must visit at least $\|y\|_1$ sites before hitting $y$.
Thus, for all $0<r<1$, $y \in \Z^d \setminus \{0\}$ and $N \in \N$ with $N \geq \|y\|_\infty$,
\begin{align*}
	\frac{d}{dr}\bigl\{ \E_r[a_N(0,y,\omega)]+(1-e^{-1})r\|y\|_1 \bigr\} \leq 0,
\end{align*}
which implies that the function in the above braces is decreasing in $r$.
Therefore, for all $0<p<q<1$, $y \in \Z^d \setminus \{0\}$ and $N \in \N$ with $N \geq \|y\|_\infty$,
\begin{align*}
	\E_p[a_N(0,y,\omega)]-\E_q[a_N(0,y,\omega)]
	\geq (1-e^{-1})(q-p)\|y\|_1.
\end{align*}
It follows by \eqref{eq:c} and the definition of the quenched Lyapunov exponent that for all $0<p<q<1$ and $x \in \Z^d \setminus \{0\}$,
\begin{align*}
	\frac{\alpha_p(x)-\alpha_q(x)}{\|x\|_1} \geq (1-e^{-1})(q-p).
\end{align*}
Note that the right side above does not depend on $x$, and $\alpha_p(\cdot)$ and $\alpha_q(\cdot)$ are norms on $\R^d$.
Consequently, the above inequality can be easily extended to all $x \in \R^d \setminus \{ 0 \}$, and the proof is complete.
\end{proof}

\section{Upper bound for the quenched Lyapunov exponent}\label{sect:qupper}
In this section, for a fixed $q_0 \in (0,1)$, we prove part~\eqref{item:diff_ql2} of Theorem~\ref{thm:diff_qlyap}.
The proof is not so long once a proposition is established (see Proposition~\ref{prop:rankone} below).

\begin{proof}[\bf Proof of part~(\ref{item:diff_ql2}) in Theorem~\ref{thm:diff_qlyap}]
Considering the proof of part~\eqref{item:diff_ql1} in Theorem~\ref{thm:diff_qlyap}, our task is to show that there exists a constant $c$ (which depends only on $d$ and $q_0$) such that for all $0<r \leq q_0$, $y \in \Z^d \setminus \{0\}$ and $N \in \N$ with $N \geq \|y\|_\infty$,
\begin{align}\label{eq:up2}
	-\frac{d}{dr}\E_r[a_N(0,y,\omega)] \leq c\|y\|_1.
\end{align}
To this end, note that for each $z \in \Z^d$ and the potential $\omega$,
\begin{align*}
	\frac{e_N(0,y,\omega_z)}{e_N(0,y,\omega)}=
	\begin{cases}
		\tilde{E}_{N,\omega}^{0,y}[e^{\ell_z(H(y))}], & \text{if } \omega(z)=1,\\[0.5em]
		\tilde{E}_{N,\omega}^{0,y}[e^{-\ell_z(H(y))}], & \text{if } \omega(z)=0,
	\end{cases}
\end{align*}
where $\omega_z$ is the potential which agrees with $\omega$ on all sites except for $z$, i.e., for $x \in \Z^d$,
\begin{align*}
\omega_z(x):=
\begin{cases}
	1-\omega(z), & \text{if } x=z,\\
	\omega(x), & \text{if } x \not=z.
\end{cases}
\end{align*}
Combining this and Proposition~\ref{prop:Russo}, one has
\begin{align}\label{eq:Russo_upper}
	-\frac{d}{dr}\E_r[a_N(0,y,\omega)]
	\leq \sum_{z \in \Z^d \cap [-N,N]^d}
	\E_r \biggl[ \biggl| \log\frac{e_N(0,y,\omega_z)}{e_N(0,y,\omega)} \biggr| \biggr],
\end{align}
and the following proposition is the key to estimating the above sum.

\begin{prop}\label{prop:rankone}
There exists a constant $\Cl{rankone}$ (which depends only on $d$ and $q_0$) such that for all $0<r \leq q_0$, $y \in \Z^d \setminus \{0\}$ and $N \in \N$ with $N \geq \|y\|_\infty$,
\begin{align*}
	\sum_{z \in \Z^d \cap [-N,N]^d}\E_r \biggl[ \biggl| \log\frac{e_N(0,y,\omega_z)}{e_N(0,y,\omega)} \biggr| \biggr]
	\leq \Cr{rankone}\,\E_r\Bigl[ \tilde{E}_{N,\omega}^{0,y}[\#\mathcal{A}(0,y)] \Bigr],
\end{align*}
where $\mathcal{A}(0,y):=\{ S_k:0 \leq k<H(y) \}$.
\end{prop}

Since the proof of the above proposition is a little long, for now let us complete the proof of part~\eqref{item:diff_ql2} in Theorem~\ref{thm:diff_qlyap}.
The same argument as in the proof of \cite[Lemma~3]{Zer98a} implies that for any $0<r \leq q_0$, $y \in \Z^d \setminus \{0\}$ and $N \in \N$ with $N \geq \|y\|_\infty$,
\begin{align}\label{eq:eanimal}
	\E_r\Bigl[ \tilde{E}_{N,\omega}^{0,y}[\#\mathcal{A}(0,y)] \Bigr]
	\leq \frac{1+\log(2d)}{-\log\{ e^{-1}+(1-e^{-1})q_0 \}}\|y\|_1.
\end{align}
From \eqref{eq:Russo_upper}, \eqref{eq:eanimal} and Proposition~\ref{prop:rankone}, we conclude that for all $0<r \leq q_0$, $y \in \Z^d \setminus \{0\}$ and $N \in \N$ with $N \geq \|y\|_\infty$,
\begin{align*}
	-\frac{d}{dr}\E_r[a_N(0,y,\omega)]
	\leq \frac{\Cr{rankone}(1+\log(2d))}{-\log\{ e^{-1}+(1-e^{-1})q_0 \}}\|y\|_1,
\end{align*}
and \eqref{eq:up2} follows.
\end{proof}

It remains to prove Proposition~\ref{prop:rankone}.
We divide the proof into two cases: Sections~\ref{ssect:high} and \ref{ssect:low} treat the high dimensional case ($d \geq 3$) and the low dimensional case ($d=1,2$), respectively.

\subsection{High dimensional case}\label{ssect:high}
This subsection is devoted to the proof of Proposition~\ref{prop:rankone} for $d \geq 3$.
First of all, we state the following lemma, which is useful to show Proposition~\ref{prop:rankone} not only for $d \geq 3$ but also for $d=1,2$.

\begin{lem}\label{lem:psi}
Let $d \geq 1$, $y \in \Z^d \setminus \{0\}$ and $N \in \N$ with $N \geq \|y\|_\infty$.
Moreover, define for any $z \in \Z^d$,
\begin{align*}
	\psi_d(z,\omega)
	:= \Biggl( 1-E^z\Biggl[ \exp\Biggl\{ -\sum_{k=1}^{H^+(z)-1}\omega(S_k) \Biggr\} \1{\{ H^+(z)<\infty \}}\Biggr] \Biggr)^{-1},
\end{align*}
where $H^+(z):=\inf\{ k>0:S_k=z \}$.
Then, for all $z \in \Z^d \cap [-N,N]^d$,
\begin{align}\label{eq:psi}
	\bigg| \log\frac{e_N(0,y,\omega_z)}{e_N(0,y,\omega)} \bigg|
	\leq 4\psi_d(z,\omega)\tilde{P}_{N,\omega}^{0,y}(H(z)<H(y)).
\end{align}
\end{lem}
\begin{proof}
We first treat the case where $z \in \Z^d \cap [-N,N]^d$ satisfies $\omega(z)=1$.
The definition of $\omega_z$ and the strong Markov property show that
\begin{align*}
	1 \leq \frac{e_N(0,y,\omega_z)}{e_N(0,y,\omega)}
	\leq 1+\frac{e_N(z,y,\omega_z)}{e_N(z,y,\omega)}\tilde{P}_{N,\omega}^{0,y}(H(z)<H(y)).
\end{align*}
We use the strong Markov property again to obtain that
\begin{align*}
	&e_N(z,y,\omega_z)\\
	&\leq \psi_d(z,\omega)E^z\Biggl[ \exp\Biggl\{ -\sum_{k=0}^{H(y)-1}\omega_z(S_k) \Biggr\}\1{\{ H(y) \leq H^+(z),\,H(y)<T([-N,N]^d) \}} \Biggr]\\
	&\leq e\psi_d(z,\omega) e_N(z,y,\omega).
\end{align*}
With these observations,
\begin{align*}
	0 \leq \log\frac{e_N(z,y,\omega_z)}{e_N(z,y,\omega)}
	\leq \log\Bigl\{ 1+e\psi_d(z,\omega)\tilde{P}_{N,\omega}^{0,y}(H(z)<H(y)) \Bigr\}.
\end{align*}
Since $t \geq \log(1+t)$ holds for $t \geq 0$, the rightmost side above is not greater than $e\psi_d(z)\tilde{P}_{N,\omega}^{0,y}(H(z)<H(y))$, and \eqref{eq:psi} follows.

We next consider the case where $z \in \Z^d \cap [-N,N]^d$ satisfies $\omega(z)=0$.
Then, the same argument as in the proof of \cite[Lemma~12]{Zer98a} shows that
\begin{align}\label{eq:Zerner}
\begin{split}
	0\leq -\log\frac{e_N(0,y,\omega_z)}{e_N(0,y,\omega)}
	&= a_N(0,y,\omega_z)-a_N(0,y,\omega)\\
	&\leq \min\Bigl\{ -\log\tilde{P}_{N,\omega}^{0,y}(H(y) \leq H(z)),1+\log\psi_d(z,\omega) \Bigr\}.
\end{split}
\end{align}
Note that the fact that $-\log t \leq 2(1-t)$ holds for $1/2 \leq t \leq 1$ implies that on the event $\{ \tilde{P}_{N,\omega}^{0,y}(H(y) \leq H(z)) \geq 1/2 \}$,
\begin{align*}
	-\log\tilde{P}_{N,\omega}^{0,y}(H(y) \leq H(z))
	\leq 2\tilde{P}_{N,\omega}^{0,y}(H(z)<H(y)).
\end{align*}
This, together with $\psi_d(z,\omega) \geq 1$, shows that the rightmost side of \eqref{eq:Zerner} is smaller than or equal to
\begin{align*}
	&2\tilde{P}_{N,\omega}^{0,y}(H(z)<H(y))\1{\{ \tilde{P}_{N,\omega}^{0,y}(H(y) \leq H(z)) \geq 1/2 \}}\\
	&+2(1+\log\psi_d(z,\omega))\tilde{P}_{N,\omega}^{0,y}(H(z)<H(y))\1{\{ \tilde{P}_{N,\omega}^{0,y}(H(y) \leq H(z))<1/2 \}}\\
	&\leq 2(\psi_d(z,\omega)+\log\psi_d(z,\omega))\tilde{P}_{N,\omega}^{0,y}(H(z)<H(y)).
\end{align*}
Therefore, \eqref{eq:psi} immediately follows from the fact that $\log t \leq t$ for $t \geq 1$.
\end{proof}

We are now in a position to prove Proposition~\ref{prop:rankone} for $d \geq 3$.

\begin{proof}[\bf Proof of Proposition~\ref{prop:rankone} for $\boldsymbol{d \geq 3}$]
Since the simple random walk is transient for $d \geq 3$, we have $P^0(H^+(0)=\infty)>0$.
Hence, for all $z \in \Z^d$,
\begin{align}\label{eq:psi3}
	\psi_d(z,\omega) \leq P^0(H^+(0)=\infty)^{-1}<\infty.
\end{align}
This together with Lemma~\ref{lem:psi} yields that for all $0<r \leq q_0$, $y \in \Z^d \setminus \{0\}$ and $N \in \N$ with $N \geq \|y\|_\infty$,
\begin{align*}
	\sum_{z \in \Z^d \cap [-N,N]^d} \E_r \biggl[ \biggl| \log\frac{e_N(0,y,\omega_z)}{e_N(0,y,\omega)} \biggr| \biggr]
	&\leq \sum_{z \in \Z^d \cap [-N,N]^d}\E_r \biggl[ 4\psi_d(z,\omega)\tilde{P}_{N,\omega}^{0,y}(H(z)<H(y)) \biggr]\\
	&\leq 4P^0(H^+(0)=\infty)^{-1} \E_r \Bigl[ \tilde{E}_{N,\omega}^{0,y}[\#\mathcal{A}(0,y)] \Bigr].
\end{align*}
Thus, the proof is complete by taking $\Cr{rankone}:=4P^0(H^+(0)=\infty)^{-1}$.
\end{proof}

\subsection{Low dimensional case}\label{ssect:low}
The aim of this section is to prove Proposition~\ref{prop:rankone} for $d=1,2$.
In this case, the strategy taken in the previous subsection does not work because the simple random walk is recurrent.
Hence, we have to estimate $\psi_d(z,\omega)$ in another way.

Let $p_c=p_c(d) \in (0,1)$ be the critical probability for independent Bernoulli site percolation on $\Z^d$, and fix $R \in 2\N$ such that
\begin{align}\label{eq:subcritical}
	q_0^{R^d}+q_0^{R^d-1}(1-q_0)R^d<p_c
\end{align}
(This is possible since $0<q_0<1$).
We now consider the boxes $\Lambda(v):=Rv+[-R/2,R/2)^d$, $v \in \Z^d$.
These boxes form a partition of $\R^d$, and denote by $\bar{z}$ the (unique) index $v$ such that $z \in \Lambda(v)$.
Furthermore, a site $v$ is said to be \emph{open} if $\Lambda(v)$ has at most one site $x$ such that $\omega(x)=1$, and \emph{closed} otherwise.
Note that, under $\P_r$, the family $(\1{\{ v \text{ is open} \}})_{v \in \Z^d}$ is the independent Bernoulli site percolation on $\Z^d$ with parameter $r^{R^d}+r^{R^d-1}(1-r)R^d$.
This site percolation induces open clusters, which are connected components of open sites.
In particular, for each $v \in \Z^d$, we denote by $\mathcal{C}_v$ the open cluster containing $v$.

Although $\psi_d(z,\omega)$ is bounded for $d \geq 3$, the following lemma says that for $d=1,2$, $\psi_d(z,\omega)$ is dominated by the size of the open cluster containing $\bar{z}$.

\begin{lem}\label{lem:cluster}
Let $d=1,2$.
Then, there exists a constant $\Cl{cluster}$ (which depends only on $d$ and $q_0$) such that for all $z \in \Z^d$,
\begin{align*}
	\psi_d(z,\omega) \leq \Cr{cluster}(1+\#\mathcal{C}_{\bar{z}}).
\end{align*}
\end{lem}
\begin{proof}
Since the simple random walk is recurrent for $d=1,2$, we have for any $z \in \Z^d$,
\begin{align*}
	\psi_d(z,\omega)
	= \Biggl( 1-E^z\Biggl[ \exp\Biggl\{ -\sum_{k=1}^{H^+(z)-1}\omega(S_k) \Biggr\} \Biggr] \Biggr)^{-1}.
\end{align*}
Let us first treat the case where $\bar{z}$ is closed.
Then, the box $\Lambda(\bar{z})$ contains at least two sites $x$ with $\omega(x)=1$.
Hence, we can find a site $x_0 \in \Lambda(\bar{z})$ such that $x_0 \not= z$ and $\omega(x_0)=1$.
It follows that
\begin{align*}
	&1-E^z\Biggl[ \exp\Biggl\{ -\sum_{k=1}^{H^+(z)-1}\omega(S_k) \Biggr\} \Biggr]\\
	&\geq 1-P^z(H^+(z) \leq H(x_0))-e^{-1}P^z(H^+(z)>H(x_0))\\
	&= (1-e^{-1})P^z(H^+(z)>H(x_0))
		\geq (1-e^{-1})\Bigl( \frac{1}{2d} \Bigr)^{dR}.
\end{align*}
Thus, since $\#\mathcal{C}_{\bar{z}}=0$ in the case where $\bar{z}$ is closed, one has
\begin{align*}
	\psi_d(z,\omega) \leq \frac{(2d)^{dR}}{1-e^{-1}}
	= \frac{(2d)^{dR}}{1-e^{-1}}(1+\#\mathcal{C}_{\bar{z}}).
\end{align*}

We next treat the case where $\bar{z}$ is open.
Then, the same argument as above does not work since the box $\Lambda(\bar{z})$ may not contain any site $x$ with $x \not= z$ and $\omega(x)=1$ (This situation actually occurs when $\omega(z)=1$ and $\omega(x)=0$ for $x \in \Lambda(\bar{z}) \setminus \{z\}$).
To overcome this problem, let us introduce the region $\mathcal{O}_z:=\bigcup_{v \in \mathcal{C}_{\bar{z}}} \Lambda(v)$ and the stopping time
\begin{align*}
	\sigma:=\inf\{ k \geq T(\mathcal{O}_z): \omega(S_k)=1 \},
\end{align*}
where $T(\mathcal{O}_z)$ is the exit time of the simple random walk from $\mathcal{O}_z$ (see \eqref{eq:exit}).
The same computation as in the first case gives
\begin{align}\label{eq:2dim1}
	1-E^z\Biggl[ \exp\Biggl\{ -\sum_{k=1}^{H^+(z)-1}\omega(S_k) \Biggr\} \Biggr]
	\geq (1-e^{-1})P^z(\sigma<H^+(z)).
\end{align}
Noting that the site $v:=\bar{S_{T(\mathcal{O}_z)}}$ is closed and $z \not\in \Lambda(v)$, the event $\{ \sigma<H^+(z) \}$ occurs if after exiting $\mathcal{O}_z$, the simple random walk consecutively visits all the sites in $\Lambda(v)$.
Hence, the strong Markov property shows that
\begin{align}\label{eq:2dim2}
	P^z(\sigma<H^+(z))
	\geq \Bigl( \frac{1}{2d} \Bigr)^{2R^d}P^z(T(\mathcal{O}_z)<H^+(z)).
\end{align}
For the last probability, we use the following estimate on the probability that the simple random walk exits an $\ell^1$-ball before revisiting its starting point (see for instance \cite[(1.20), (1.38) and Theorem~{1.6.6}]{Law91_book}): There exists a constant $c \geq 1$ (which depends only on $d$) such that
\begin{align}\label{eq:Lawler}
	P^0(\tau<H^+(0))^{-1} \leq c \times
	\begin{cases}
		\#\mathcal{O}_z, & \text{if } d=1,\\
		\log(\#\mathcal{O}_z), & \text{if } d=2,
	\end{cases}
\end{align}
where $\tau$ is the exit time for the simple random walk from the $\ell^1$-ball of radius $\#\mathcal{O}_z$ and center $0$, i.e.,
\begin{align*}
	\tau:=\inf\{ k \geq 0:\|S_k\|_1>\#\mathcal{O}_z \}.
\end{align*}
Noting that $P^0(\tau<H^+(0)) \leq P^z(T(\mathcal{O}_z)<H^+(z))$ and $\#\mathcal{O}_z=R^d(\#\mathcal{C}_{\bar{z}})$, we have from \eqref{eq:2dim1}, \eqref{eq:2dim2} and \eqref{eq:Lawler},
\begin{align*}
	\psi_d(z,\omega)
	\leq \frac{c(2d)^{2R^d}}{1-e^{-1}}(\#\mathcal{O}_z)
	\leq \frac{c(2d)^{2R^d}R^d}{1-e^{-1}} (1+\#\mathcal{C}_{\bar{z}}).
\end{align*}
Consequently, the lemma follows by taking $\Cr{cluster}:=c(2d)^{2R^d}R^d/(1-e^{-1})$.
\end{proof}

We now turn to the proof of Proposition~\ref{prop:rankone} for $d=1,2$.

\begin{proof}[\bf Proof of Proposition~\ref{prop:rankone} for $\boldsymbol{d=1,2}$]
Lemmata~\ref{lem:psi} and \ref{lem:cluster} imply that for all $0<r \leq q_0$, $y \in \Z^d \setminus \{0\}$ and $N \in \N$ with $N \geq \|y\|_\infty$,
\begin{align*}
	&\sum_{z \in \Z^d \cap [-N,N]^d} \E_r \biggl[ \biggl| \log\frac{e_N(0,y,\omega_z)}{e_N(0,y,\omega)} \biggr| \biggr]\\
	&\leq 4\Cr{cluster} \sum_{z \in \Z^d \cap [-N,N]^d} \E_r\Bigl[ (1+\#\mathcal{C}_{\bar{z}}) \tilde{P}_{N,\omega}^{0,y}(H(z)<H(y)) \Bigr]\\
	&= 4\Cr{cluster} \Biggl( \E_r\Bigl[ \tilde{E}_{N,\omega}^{0,y}[\#\mathcal{A}(0,y)] \Bigr]
		+\E_r\Biggl[ \tilde{E}_{N,\omega}^{0,y} \Biggl[ \sum_{z \in \mathcal{A}(0,y)} \#\mathcal{C}_{\bar{z}} \Biggr] \Biggr] \Biggr).
\end{align*}
Hence, it suffices to prove that there exists a constant $\Cl{animal}$ (which depends only on $d$ and $q_0$) such that all $0<r \leq q_0$, $y \in \Z^d \setminus \{0\}$ and $N \in \N$ with $N \geq \|y\|_\infty$,
\begin{align}\label{eq:FN}
	\E_r\Biggl[ \tilde{E}_{N,\omega}^{0,y} \Biggl[ \sum_{z \in \mathcal{A}(0,y)} \#\mathcal{C}_{\bar{z}} \Biggr] \Biggr]
	\leq \Cr{animal}\E_r\Bigl[ \tilde{E}_{N,\omega}^{0,y}[\#\mathcal{A}(0,y)] \Bigr].
\end{align}

To show \eqref{eq:FN}, we state some results for lattice animals on $\Z^d$ (which are finite connected subsets of $\Z^d$).
For $n \geq 1$, denote by $\mathbb{A}_n$ the set of all lattice animals, of size $n$, containing $0$.
Moreover, let $(\tilde{\Gamma}_v)_{v \in \Z^d}$ be i.i.d.~random lattice animals with the common law $\P_{q_0}(\mathcal{C}_0 \in \cdot)$ (We write $P$ for the probability measure governing $(\tilde{\Gamma}_v)_{v \in \Z^d}$).
Then, due to \eqref{eq:subcritical}, the following lemma is an immediate consequence of \cite[(2.6) and (2.7)]{FonNew93} and \cite[page~183]{Sch14}.

\begin{lem}\label{lem:GLA}
The following results hold:
\begin{enumerate}
	\item\label{item:GLA2}
		We have for all $n \geq 1$ and $t \geq 0$,
		\begin{align*}
			\P_{q_0}\Biggl( \sup_{\Gamma \in \mathbb{A}_n} \frac{1}{\#\Gamma}\sum_{v \in \Gamma} \#\mathcal{C}_v \geq t \Bigg)
			\leq P\Biggl( \sup_{\Gamma \in \bigcup_{m=n}^\infty \mathbb{A}_m} \frac{1}{\#\Gamma} \sum_{v \in \Gamma} (\#\tilde{\Gamma}_v)^2 \geq \frac{t}{2} \Bigg).
		\end{align*}
	\item\label{item:GLA1}
		There exist constants $\Cl{Martin1}$ and $\Cl{Martin2}$ (which depend only on $d$ and $q_0$) such that for all $n \geq 1$,
		\begin{align*}
			P\Biggl( \sup_{\Gamma \in \mathbb{A}_n} \frac{1}{\#\Gamma} \sum_{v \in \Gamma}(\#\tilde{\Gamma}_v)^2 \geq \Cr{Martin1} \Biggr)
			\leq \Cr{Martin1}\exp\{-\Cr{Martin2}n^{1/5}\}.
		\end{align*}
\end{enumerate}
\end{lem}

Define for any lattice animal $\Gamma$ on $\Z^d$,
\begin{align*}
	\mathcal{E}(\Gamma):=\biggl\{ \frac{1}{\#\Gamma} \sum_{v \in \Gamma} \#\mathcal{C}_v \geq 3\Cr{Martin1} \biggr\}.
\end{align*}
Moreover, for each $y \in \Z^d \setminus \{0\}$, set
\begin{align*}
	\bar{\mathcal{A}}(0,y):=\{ \bar{z}:z \in \mathcal{A}(0,y) \},
\end{align*}
which is a lattice animal on $\Z^d$ containing $0$.
Then, for all $0<r \leq q_0$, $y \in \Z^d \setminus \{0\}$ and $N \in \N$ with $N \geq \|y\|_\infty$, the left side of \eqref{eq:FN} is smaller than or equal to
\begin{align*}
	&R^d\E_r\Biggl[ \tilde{E}_{N,\omega}^{0,y} \Biggl[ \sum_{v \in \bar{\mathcal{A}}(0,y)} \#\mathcal{C}_v \Biggr] \Biggr]\\
	&\leq R^d \Biggl\{ 2\Cr{Martin1}\E_r\Bigl[ \tilde{E}_{N,\omega}^{0,y}[\#\bar{\mathcal{A}}(0,y)] \Bigr]
		+\E_r\Biggl[ \tilde{E}_{N,\omega}^{0,y} \Biggl[ \Biggl( \sum_{v \in \bar{\mathcal{A}}(0,y)} \#\mathcal{C}_v \Biggr) \1{\mathcal{E}(\bar{\mathcal{A}}(0,y))} \Biggr] \Biggr] \Biggr\}.
\end{align*}
Since $\#\bar{\mathcal{A}}(0,y) \leq \#\mathcal{A}(0,y)$, our task is now to prove that there exists a constant $c$ (which depends only on $d$ and $q_0$) such that for all $0<r \leq q_0$, $y \in \Z^d \setminus \{0\}$ and $N \in \N$ with $N \geq \|y\|_\infty$,
\begin{align}\label{eq:harmless}
	\E_r\Biggl[ \tilde{E}_{N,\omega}^{0,y} \Biggl[ \Biggl( \sum_{v \in \bar{\mathcal{A}}(0,y)} \#\mathcal{C}_v \Biggr) \1{\mathcal{E}(\bar{\mathcal{A}}(0,y))} \Biggr] \Biggr]
	\leq c.
\end{align}
Indeed, once \eqref{eq:harmless} is proved, the left side of \eqref{eq:FN} is bounded from above by
\begin{align*}
	(2\Cr{Martin1}+c)R^d\E_r\Bigl[ \tilde{E}_{N,\omega}^{0,y}[\#\mathcal{A}(0,y)] \Bigr],
\end{align*}
and \eqref{eq:FN} follows by taking $\Cr{animal}:=(2\Cr{Martin1}+c)R^d$.

For any $0<r \leq q_0$, $y \in \Z^d \setminus \{0\}$ and $N \in \N$ with $N \geq \|y\|_\infty$, the left side in \eqref{eq:harmless} is equal to
\begin{align*}
	&\sum_{n=1}^\infty \E_r\Biggl[ \tilde{E}_{N,\omega}^{0,y} \Biggl[ \Biggl( \sum_{v \in \bar{\mathcal{A}}(0,y)} \#\mathcal{C}_v \Biggr) \1{\mathcal{E}(\bar{\mathcal{A}}(0,y))} \1{\{ \#\bar{\mathcal{A}}(0,y)=n \}} \Biggr] \Biggr]\\
	&\leq \sum_{n=1}^\infty \sum_{v \in [-n,n]^d} \E_r\bigl[ \#\mathcal{C}_v \1{\bigcup_{\Gamma \in \mathbb{A}_n} \mathcal{E}(\Gamma)} \bigr].
\end{align*}
Note that $\E_r\bigl[ \#\mathcal{C}_v \1{\bigcup_{\Gamma \in \mathbb{A}_n} 
\mathcal{E}(\Gamma)} \bigr]$ is increasing in $r$.
This combined with the Cauchy--Schwarz inequality implies that for all $n \geq 1$ and $v \in [-n,n]^d$,
\begin{align*}
	\E_r\bigl[ \#\mathcal{C}_v \1{\bigcup_{\Gamma \in \mathbb{A}_n} \mathcal{E}(\Gamma)} \bigr]
	&\leq \E_{q_0}\bigl[ \#\mathcal{C}_v \1{\bigcup_{\Gamma \in \mathbb{A}_n} \mathcal{E}(\Gamma)} \bigr]\\
	&\leq \E_{q_0}\bigl[ \#\mathcal{C}_0^2 \bigr]^{1/2} \P_{q_0}\Biggl( \sup_{\Gamma \in \mathbb{A}_n} \frac{1}{\#\Gamma} \sum_{v' \in \Gamma}\#\mathcal{C}_{v'} \geq 3\Cr{Martin1} \Biggr)^{1/2}.
\end{align*}
Lemma~\ref{lem:GLA} says that there exists a constant $c'$ (which depends only on $d$ and $q_0$) such that for all $n \geq 1$,
\begin{align*}
	\P_{q_0}\Biggl( \sup_{\Gamma \in \mathbb{A}_n} \frac{1}{\#\Gamma} \sum_{v' \in \Gamma}\#\mathcal{C}_{v'} \geq 2\Cr{Martin1} \Biggr)
	&\leq P\Biggl( \sup_{\Gamma \in \bigcup_{m=n}^\infty \mathbb{A}_m} \frac{1}{\#\Gamma} \sum_{v' \in \Gamma} (\#\tilde{\Gamma}_{v'})^2 \geq \Cr{Martin1} \Bigg)\\
	&\leq \sum_{m=n}^\infty P\Biggl( \sup_{\Gamma \in \mathbb{A}_m} \frac{1}{\#\Gamma} \sum_{v \in \Gamma} (\#\tilde{\Gamma}_v)^2 \geq \Cr{Martin1} \Bigg)\\
	&\leq c'\exp\Bigl\{ -\frac{\Cr{Martin2}}{2}n^{1/5} \Bigr\}.
\end{align*}
With these observations, we can derive \eqref{eq:harmless} as follows:
\begin{align*}
	&\E_r\Biggl[ \tilde{E}_{N,\omega}^{0,y} \Biggl[ \Biggl( \sum_{v \in \bar{\mathcal{A}}(0,y)} \#\mathcal{C}_v \Biggr) \1{\mathcal{E}(\bar{\mathcal{A}}(0,y))} \Biggr] \Biggr]\\
	&\leq (c'\E_{q_0}\bigl[ \#\mathcal{C}_0^2 \bigr])^{1/2} \sum_{n=1}^\infty (2n+1)^d \exp\Bigl\{ -\frac{\Cr{Martin2}}{4}n^{1/5} \Bigr\}<\infty.
\end{align*}
Therefore, the proof of Proposition~\ref{prop:rankone} is complete for $d=1,2$.
\end{proof}

\section{Bounds for the annealed Lyapunov exponent}\label{sect:abound}
This section is devoted to the proof of Theorem~\ref{thm:diff_alyap}.
We fix $q_0$ with $0<q_0<1$.
For each $0 \leq r \leq 1$, the path measure $\tilde{\P}_r^{0,y}$ is defined by
\begin{align}\label{eq:apathmeas}
	\frac{d\tilde{\P}_r^{0,y}}{dP^0}
	= \E_r[e(0,y,\omega)]^{-1} \E_r\Biggl[ \exp\Biggl\{ -\sum_{k=0}^{H(y)-1}\omega(S_k) \Biggr\}\Biggr] \1{\{ H(y)<\infty \}},
\end{align}
and $\tilde{\E}_r^{0,y}$ is the expectation with respect to $\tilde{\P}_r^{0,y}$.
Then, the following two lemmata are the key to proving Theorem~\ref{thm:diff_alyap}.

\begin{lem}\label{lem:derivative_b}
We have for all $0<r<1$ and $y \in \Z^d \setminus \{0\}$,
\begin{align*}
	-\frac{d}{dr}b_r(0,y)
	= \tilde{\E}_r^{0,y}\Biggl[ \sum_{\substack{z \in \Z^d\\ \ell_z(H(y)) \geq 1}} \frac{1-e^{-\ell_z(H(y))}}{r+e^{-\ell_z(H(y))}(1-r)} \Biggr].
\end{align*}
\end{lem}

\begin{lem}\label{lem:l_bound}
If $d \geq 3$, then there exists a constant $\Cl{an_GLA}$ (which depends only on $d$ and $q_0$) such that for all $0<r \leq q_0$ and $y \in \Z^d \setminus \{0\}$,
\begin{align*}
	\tilde{\E}_r^{0,y}\Biggl[ \sum_{\substack{z \in \Z^d\\ \ell_z(H(y)) \geq 1}} \frac{1-e^{-\ell_z(H(y))}}{r+e^{-\ell_z(H(y))}(1-r)} \Biggr]
	\leq \Cr{an_GLA}\|y\|_1.
\end{align*}
\end{lem}

Lemma~\ref{lem:derivative_b} gives the derivative of $b_r(0,y)$ at $r$ by using the path measure $\tilde{\P}_r^{0,y}$.
Moreover, Lemma~\ref{lem:l_bound} guarantees that in the case $d \geq 3$, the derivative of $b_r(0,y)$ at $r$ can be bounded from below uniformly in $r \in (0,q_0]$.
Let us show Theorem~\ref{thm:diff_alyap} before proving the lemmata above.

\begin{proof}[\bf Proof of Theorem~\ref{thm:diff_alyap}]
We first prove part~\eqref{item:diff_al1} of Theorem~\ref{thm:diff_alyap}, which gives lower bounds for differences between annealed Lyapunov exponents.
Note that for each $0<r<1$, the function $f(t):=(1-t)/\{ r+t(1-r) \}$ is decreasing in $t \in [0,1]$.
Hence, Lemma~\ref{lem:derivative_b} and the fact that $\#\mathcal{A}(0,y) \geq \|y\|_1$ holds $\tilde{\P}_r^{0,y}$-a.s.~yield that for all $0<r<1$ and $y \in \Z^d \setminus \{0\}$,
\begin{align*}
	-\frac{d}{dr}b_r(0,y)
	&= \tilde{\E}_r^{0,y}\Biggl[ \sum_{\substack{z \in \Z^d\\ \ell_z(H(y)) \geq 1}} \frac{1-e^{-\ell_z(H(y))}}{r+e^{-\ell_z(H(y))}(1-r)} \Biggr]\\
	&\geq \frac{1-e^{-1}}{r+e^{-1}(1-r)}\,\tilde{\E}_r^{0,y}[\#\mathcal{A}(0,y)]\\
	&\geq (1-e^{-1})\|y\|_1.
\end{align*}
Thus, similar to the proof of part~\eqref{item:diff_al1} in Theorem~\ref{thm:diff_qlyap}, one has for all $0<p<q<1$,
\begin{align*}
	\inf_{x \in \R^d\setminus\{0\}}\frac{\beta_p(x)-\beta_q(x)}{\|x\|_1}
	\geq (1-e^{-1})(q-p),
\end{align*}
which is the desired lower bound for the difference between annealed Lyapunov exponents.

Let us next show part~\eqref{item:diff_al2} of Theorem~\ref{thm:diff_alyap}, which gives upper bounds for differences between annealed Lyapunov exponents.
Lemma~\ref{lem:derivative_b} and the monotonicity of the function $f(t)=(1-t)/\{ r+t(1-r) \}$, $0 \leq t \leq 1$, tell us that for all $0<r \leq q_0$ and $y \in \Z^d \setminus \{0\}$,
\begin{align}\label{eq:b_low}
\begin{split}
	-\frac{d}{dr}b_r(0,y)
	&= \tilde{\E}_r^{0,y}\Biggl[ \sum_{\substack{z \in \Z^d\\ \ell_z(H(y)) \geq 1}} \frac{1-e^{-\ell_z(H(y))}}{r+e^{-\ell_z(H(y))}(1-r)} \Biggr]\\
	&\leq r^{-1}\tilde{\E}_r^{0,y}[\#\mathcal{A}(0,y)].
\end{split}
\end{align}
Moreover, by the same arguments as in \cite[Lemma~3]{Zer98a}, it holds that for all $0<r<1$ and $y \in \Z^d \setminus \{0\}$,
\begin{align}\label{eq:qLA}
	\tilde{\E}_r^{0,y}[\#\mathcal{A}(0,y)]
	\leq \frac{1+\log(2d)}{-\log\{ e^{-1}+(1-e^{-1})r \}}\|y\|_1.
\end{align}
Hence, \eqref{eq:b_low} and \eqref{eq:qLA} imply that for all $0<r \leq q_0$ and $y \in \Z^d \setminus \{0\}$,
\begin{align*}
	-\frac{d}{dr}b_r(0,y)
	\leq \frac{1+\log(2d)}{-\log\{ e^{-1}+(1-e^{-1})q_0 \}}\,r^{-1}\|y\|_1.
\end{align*}
Therefore, by taking
\begin{align*}
	\Cr{alyap}:=\frac{1+\log(2d)}{-\log\{ e^{-1}+(1-e^{-1})q_0 \}},
\end{align*}
one sees immediately that for all $0<p<q \leq q_0$,
\begin{align*}
	\sup_{x \in \R^d \setminus \{0\}}	\frac{\beta_p(x)-\beta_q(x)}{\|x\|_1} \leq \Cr{alyap}(\log q-\log p).
\end{align*}
In particular, if $d \geq 3$, then Lemma~\ref{lem:l_bound} allows us to obtain the following bound for the derivative of $b_r(0,y)$ at $r$ instead of \eqref{eq:b_low}:
For all $0<r \leq q_0$ and $y \in \Z^d \setminus \{0\}$,
\begin{align*}
	-\frac{d}{dr}b_r(0,y)
	= \tilde{\E}_r^{0,y}\Biggl[ \sum_{\substack{z \in \Z^d\\ \ell_z(H(y)) \geq 1}}
		\frac{1-e^{-\ell_z(H(y))}}{r+e^{-\ell_z(H(y))}(1-r)} \Biggr]
	\leq \Cr{an_GLA}\|y\|_1.
\end{align*}
It follows immediately that
\begin{align*}
	\sup_{x \in \R^d \setminus \{0\}}	\frac{\beta_p(x)-\beta_q(x)}{\|x\|_1} \leq \Cr{an_GLA}(q-p),
\end{align*}
and the proof of part~\eqref{item:diff_al2} in Theorem~\ref{thm:diff_alyap} is complete.
\end{proof}

We close this section with the proofs of Lemmata~\ref{lem:derivative_b} and \ref{lem:l_bound}.

\begin{proof}[\bf Proof of Lemma~\ref{lem:derivative_b}]
First of all, we often add $\S$ to the notation of $H(y)$ and $\ell_z(N)$ to clarify the dependency of the trajectory $\S$ of the simple random walk: $H(y)=H(y,\S)$ and $\ell_z(N)=\ell_z(N,\S)$.
Fix $y \in \Z^d \setminus \{0\}$, and define for $0<r<1$ and the trajectory $\S$ of the simple random walk,
\begin{align*}
	\phi(r,\S)
	&:= \E_r\Biggl[ \exp\Biggl\{ -\sum_{k=0}^{H(y,\S)-1}\omega(S_k) \Biggr\}\Biggr] \1{\{ H(y,\S)<\infty \}}\\
	&= \prod_{\substack{z \in \Z^d\\ \ell_z(H(y,\S),\S) \geq 1}} \E_r\bigl[ e^{-\ell_z(H(y,\S),\S)\omega(0)} \bigr] \1{\{ H(y,\S)<\infty \}}.
\end{align*}
It follows from Fubini's theorem that $b_r(0,y)=-\log E^0[\phi(r,\S)]$.
Hence, our task is to show that
\begin{align}\label{eq:diff_phi}
	\frac{d}{dr} E^0[\phi(r,\S)]
	= E^0\Biggl[ \phi(r,\S) \sum_{\substack{z \in \Z^d\\ \ell_z(H(y,\S),\S) \geq 1}} \frac{1-e^{-\ell_z(H(y,\S),\S)}}{r+e^{-\ell_z(H(y,\S),\S)}(1-r)} \Biggr].
\end{align}
Indeed, once \eqref{eq:diff_phi} is proved, we have
\begin{align*}
	\frac{d}{dr}b_r(0,y)
	&= -E^0[\phi(r,\S)]^{-1} E^0\Biggl[ \phi(r,\S) \sum_{\substack{z \in \Z^d\\ \ell_z(H(y,\S),\S) \geq 1}} \frac{1-e^{-\ell_z(H(y,\S),\S)}}{r+e^{-\ell_z(H(y,\S),\S)}(1-r)} \Biggr]\\
	&= -\tilde{\E}_r^{0,y}\Biggl[ \sum_{\substack{z \in \Z^d\\ \ell_z(H(y)) \geq 1}} \frac{1-e^{-\ell_z(H(y))}}{r+e^{-\ell_z(H(y))}(1-r)} \Biggr],
\end{align*}
which is the desired conclusion.

To prove \eqref{eq:diff_phi}, note that for the trajectory $\S$ of the simple random walk,
\begin{align*}
	\E_r\bigl[ e^{-\ell_z(H(y,\S),\S)\omega(0)} \bigr]
	= r+e^{-\ell_z(H(y,\S),\S)}(1-r).
\end{align*}
This tells us that $\phi(r,\S)$ is differentiable at $r \in (0,1)$ with the derivative
\begin{align}\label{eq:derivative_phi}
	\phi_r(r,\S)
	= \phi(r,\S) \sum_{\substack{z \in \Z^d\\ \ell_z(H(y,\S),\S) \geq 1}}
	\frac{1-e^{-\ell_z(H(y,\S),\S)}}{r+e^{-\ell_z(H(y,\S),\S)}(1-r)}.
\end{align}
Therefore, differentiating under the integral sign (the possibility of this operation will be checked in Appendix~\ref{app:diff_int}), one has for $0<r<1$,
\begin{align*}
	\frac{d}{dr} E^0[\phi(r,\S)]
	&=E^0[\phi_r(r,\S)]\\
	&= E^0\Biggl[ \phi(r,\S) \sum_{\substack{z \in \Z^d\\ \ell_z(H(y,\S),\S) \geq 1}} \frac{1-e^{-\ell_z(H(y,\S),\S)}}{r+e^{-\ell_z(H(y,\S),\S)}(1-r)} \Biggr],
\end{align*}
and \eqref{eq:diff_phi} follows.
\end{proof}

\begin{proof}[\bf Proof of Lemma~\ref{lem:l_bound}]
Let $d \geq 3$ and fix $0<r<1$ and $y \in \Z^d \setminus \{0\}$.
Furthermore, for each $z \in \Z^d$, set $H_1(z):=H(z)$ and define inductively,
\begin{align*}
	H_{\ell+1}(z):=\inf\{ k>H_\ell(z):S_k=z \},\qquad \ell \geq 1.
\end{align*}
Then, one has
\begin{align}\label{eq:rho}
\begin{split}
	&\tilde{\E}_r^{0,y}\Biggl[ \sum_{\substack{z \in \Z^d\\ \ell_z(H(y)) \geq 1}} \frac{1-e^{-\ell_z(H(y))}}{r+e^{-\ell_z(H(y))}(1-r)} \Biggr]\\
	&= \sum_{z \in \Z^d \setminus \{y\}} \E_r[e(0,y,\omega)]^{-1} \sum_{\ell=1}^\infty \frac{1-e^{-\ell}}{r+e^{-\ell}(1-r)} \,\E_r[\rho_\ell(z,\omega)],
\end{split}
\end{align}
where
\begin{align*}
	\rho_\ell(z,\omega)
	:= E^0\Biggl[ \exp\Biggl\{ -\sum_{k=0}^{H(y)-1}\omega(S_k) \Biggr\} \1{\{ H_\ell(z)<H(y)<H_{\ell+1}(z) \}} \Biggr].	
\end{align*}
The strong Markov property gives that
\begin{align*}
	\rho_\ell(z,\omega)
	&\leq E^z\Biggl[ \exp\Biggl\{ -\sum_{k=0}^{H^+(z)-1}\omega(S_k) \Biggr\} \1{\{ H^+(z)<\infty \}} \Biggr]^{\ell-1}\\
	&\quad \times
		E^0\Biggl[ \exp\Biggl\{ -\sum_{k=0}^{H(y)-1}\omega(S_k) \Biggr\} \1{\{ H(z)<H(y)<H_2(z) \}} \Biggr]
\end{align*}
(see the statement of Lemma~\ref{lem:psi} for the notation $H^+(z)$).
Hence, in the case where $\omega(z)=0$, $\rho_\ell(z,\omega)$ is bounded from above by
\begin{align}\label{eq:indep1}
\begin{split}
	&P^0(H^+(0)<\infty)^{\ell-1}\\
	&\times E^0\Biggl[ \exp\Biggl\{ -\sum_{\substack{0 \leq k<H(y)\\ S_k \not= z}}\omega(S_k) \Biggr\} \1{\{ H(z)<H(y)<H_2(z) \}} \Biggr].
\end{split}
\end{align}
On the other hand, in the case where $\omega(z)=1$, $\rho_\ell(z,\omega)$ is smaller than or equal to
\begin{align}\label{eq:indep2}
\begin{split}
	&e^{-\ell} P^0(H^+(0)<\infty)^{\ell-1}\\
	&\times E^0\Biggl[ \exp\Biggl\{ -\sum_{\substack{0 \leq k<H(y)\\ S_k \not= z}}\omega(S_k) \Biggr\} \1{\{ H(z)<H(y)<H_2(z) \}} \Biggr].
\end{split}
\end{align}
Note that both \eqref{eq:indep1} and \eqref{eq:indep2} are independent of $\omega(z)$, and the simple random walk visits only one time before hitting $y$ on the event $\{ H(z)<H(y)<H_2(z) \}$.
It follows that
\begin{align*}
	\E_r[\rho_\ell(z,\omega)]
	&\leq e\{ r+e^{-\ell}(1-r) \}P^0(H^+(0)<\infty)^{\ell-1}\\
	&\quad \times \E_r \Biggl[ E^0\Biggl[ \exp\Biggl\{ -\sum_{k=0}^{H(y)-1}\omega(S_k) \Biggr\} \1{\{ H(z)<H(y)<H_2(z) \}} \Biggr] \Biggr].
\end{align*}
Therefore, \eqref{eq:qLA}, \eqref{eq:rho} and the fact that $P^0(H^+(0)<\infty)<1$ holds for $d \geq 3$ yield that for all $0<r \leq q_0$ and $y \in \Z^d \setminus \{0\}$,
\begin{align*}
	&\tilde{\E}_r^{0,y}\Biggl[ \sum_{\substack{z \in \Z^d\\ \ell_z(H(y)) \geq 1}} \frac{1-e^{-\ell_z(H(y))}}{r+e^{-\ell_z(H(y))}(1-r)} \Biggr]\\
	&\leq e\sum_{\ell=1}^\infty P^0(H^+(0)<\infty)^{\ell-1}\,\tilde{\E}_r^{0,y}[\#\mathcal{A}(0,y)]\\
	&\leq eP^0(H^+(0)=\infty)^{-1}\frac{1+\log(2d)}{-\log\{ e^{-1}+(1-e^{-1})q_0 \}}\|y\|_1,
\end{align*}
which is the desired conclusion.
\end{proof}

\section{Comment on the large deviation principle}\label{sect:LDP}
In this section, we discuss the large deviation principle for the simple random walk in the Bernoulli potential.
Let $0 \leq r \leq 1$ and let $\omega$ be the Bernoulli potential with parameter $r$.
Consider the path measures $Q_{n,\omega}^\textrm{qu}$ and $Q_{n,r}^\textrm{an}$ defined as follows:
\begin{align*}
	\frac{dQ_{n,\omega}^\mathrm{qu}}{dP^0}
	= \frac{1}{Z_{n,\omega}^\mathrm{qu}} \exp\biggl\{ -\sum_{k=0}^{n-1}\omega(S_k) \biggr\}
\end{align*}
and
\begin{align*}
	\frac{dQ_{n,r}^\mathrm{an}}{dP^0}
	= \frac{1}{Z_{n,r}^\mathrm{an}} \E_r\biggl[ \exp\biggl\{ -\sum_{k=0}^{n-1}\omega(S_k) \biggr\} \biggr],
\end{align*}
where $Z_{n,\omega}^\mathrm{qu}$ and $Z_{n,r}^\mathrm{an}$ are the corresponding normalizing constants.
Moreover, for $\lambda \geq 0$, write $\alpha_r(\lambda,\cdot)$ and $\beta_r(\lambda,\cdot)$ for the quenched and annealed Lyapunov exponents for the potential $\omega+\lambda=(\omega(x)+\lambda)_{x \in \Z^d}$, respectively.
Note that $\alpha_r(\lambda,x)$ and $\beta_r(\lambda,x)$ are continuous in $(\lambda,x) \in [0,\infty) \times \R^d$ and concave increasing in $\lambda$ (see \cite[Theorem~A]{Flu07}, \cite[Theorem~1.1]{Mou12} and \cite[Proposition~4]{Zer98a}).
Then, set for $x \in \R^d$,
\begin{align*}
	I_r(x):=\sup_{\lambda \geq 0}\{ \alpha_r(\lambda,x)-\lambda \}
\end{align*}
and
\begin{align*}
	J_r(x):=\sup_{\lambda \geq 0}\{ \beta_r(\lambda,x)-\lambda \}.
\end{align*}

The following proposition states the large deviation principles for the simple random walk in the Bernoulli potential, which is a direct application of \cite[Theorem~B]{Flu07}, \cite[Theorem~{1.10}]{Mou12} and \cite[Theorem~19]{Zer98a}.

\begin{prop}\label{prop:ldp}
Let $0<r \leq 1$.
Then, the law of $S_n/n$ obeys the following quenched and annealed large deviation principles with the rate functions $I_r$ and $J_r$, respectively:
\begin{itemize}
	\item (Quenched case) $\P_r \hyphen \as$, for any Borel set $\Gamma$ in $\R^d$,
		\begin{align*}
			-\inf_{x \in \Gamma^o}I_r(x)
			&\leq \liminf_{n \to \infty} \frac{1}{n}\log Q_{n,\omega}^\mathrm{qu}(S_n \in n\Gamma)\\
			&\leq \limsup_{n \to \infty} \frac{1}{n}\log Q_{n,\omega}^\mathrm{qu}(S_n \in n\Gamma)
			\leq -\inf_{x \in \bar{\Gamma}}I_r(x).
		\end{align*}
	\item (Annealed case) For any Borel set $\Gamma$ in $\R^d$,
		\begin{align*}
			-\inf_{x \in \Gamma^o}J_r(x)
			&\leq \liminf_{n \to \infty} \frac{1}{n}\log Q_{n,r}^\mathrm{an}(S_n \in n\Gamma)\\
			&\leq \limsup_{n \to \infty} \frac{1}{n}\log Q_{n,r}^\mathrm{an}(S_n \in n\Gamma)
			\leq -\inf_{x \in \bar{\Gamma}}J_r(x).
		\end{align*}
\end{itemize}
Here $\Gamma^o$ and $\bar{\Gamma}$ are the interior and closure of $\Gamma$, respectively.
Furthermore, the rate functions $I_r$ and $J_r$ are continuous and convex on their effective domains, which are equal to the closed $\ell^1$-unit ball.
\end{prop}

Since exactly the same arguments used in the previous sections work for $\alpha_r(\lambda,\cdot)$ and $\beta_r(\lambda,\cdot)$, we can replace $\alpha_r(\cdot)$ and $\beta_r(\cdot)$ with $\alpha_r(\lambda,\cdot)$ and $\beta_r(\lambda,\cdot)$ in Theorems~\ref{thm:diff_qlyap} and \ref{thm:diff_alyap}, respectively (In particular, the constants $\Cr{qlyap}$ and $\Cr{alyap}$ can be chosen independently of $\lambda \geq 0$).
This derives the following differences between quenched and annealed rate functions.

\begin{cor}\label{cor:diff_rate}
Let $d \geq 1$ and $0<q_0<1$.
Then, there exist constants $\Cl{qrate}$ and $\Cl{arate}$ (which depend only on $d$ and $q_0$) such that the following results hold:
\begin{itemize}
	\item(Quenched case)
		For all $0<p<q<1$,
		\begin{align}\label{eq:l_qrate}
			\inf_{x \in \R^d \setminus \{0\}}	\frac{I_p(x)-I_q(x)}{\|x\|_1} \geq (1-e^{-1})(q-p),
		\end{align}
		and for all $0<p<q \leq q_0$,
		\begin{align*}
			\sup_{x \in \R^d \setminus \{0\}}	\frac{I_p(x)-I_q(x)}{\|x\|_1}
			\leq \Cr{qrate}(q-p).
		\end{align*}
	\item(Annealed case)
		For all $0<p<q<1$, we have \eqref{eq:l_qrate} with $I_p(\cdot)$ and $I_q(\cdot)$ replaced by $J_p(\cdot)$ and $J_q(\cdot)$, respectively.
		Moreover, for all $0<p<q \leq q_0$,
		\begin{align*}
			\sup_{x \in \R^d \setminus \{0\}}	\frac{J_p(x)-J_q(x)}{\|x\|_1}
			\leq \Cr{arate}(\log q-\log p).
		\end{align*}
		In particular, if $d \geq 3$, then the right side above can be replaced by $\Cr{arate}(q-p)$.
\end{itemize}
\end{cor}
\begin{proof}
We treat only the upper bound for the quenched case since the same argument works for the other cases.
Set for any $0<r<1$ and $x \in \R^d$,
\begin{align*}
	\lambda_r^{\rm qu}(x):=\inf\{ \lambda>0:\partial_- \alpha_r(\lambda,x) \leq 1 \},
\end{align*}
where $\partial_-\alpha_r(\lambda,x)$ is the left derivative of $\alpha_r(\lambda,x)$ with respect to $\lambda$ (Note that $\lambda_r^{\rm qu}(x)$ may be equal to $\infty$).
Clearly, $\lambda_r^{\rm qu}(x)$ attains the supremum in the definition of $I_r(x)$.
Hence, Theorem~\ref{thm:diff_qlyap} with $\alpha_r(\cdot)$ replaced by $\alpha_r(\lambda,\cdot)$ implies that for any $0<p<q<1$, $t \geq 0$ and $x \in \R^d \setminus \{0\}$,
\begin{align*}
	&\frac{I_p(x)-\{ \alpha_q(\lambda_q^{\rm qu}(x) \wedge t,x)-(\lambda_q^{\rm qu}(x) \wedge t) \}}{\|x\|_1}\\
	&\geq \frac{\alpha_p(\lambda_q^{\rm qu}(x) \wedge t,x)-\alpha_q(\lambda_q^{\rm qu}(x) \wedge t,x)}{\|x\|_1}
	\geq (1-e^{-1})(q-p).
\end{align*}
Since $\alpha_q(\lambda,x)$ is continuous in $\lambda$, letting $t \to \infty$ proves \eqref{eq:l_qrate}.
\end{proof}

\appendix
\section{Differentiation under the integral sign}\label{app:diff_int}
The aim of this section is to discuss differentiation under the integral sign in the proof of Lemma~\ref{lem:derivative_b}.
We first mention differentiation under the integral sign in measure theory (see for instance \cite[Theorem~{6.28}]{Kle13_book}).

\begin{lem}\label{lem:difflem}
Let $I$ be a nonempty open interval of $\R$ and let $\Sigma$ be a measure space equipped with a measure $\mu$.
Suppose that $f:I \times \Sigma \to \R$ is a function satisfying the following conditions:
\begin{enumerate}
	\item\label{item:difflem1}
		For any $r \in I$, the function $\sigma \mapsto f(r,\sigma)$ is $\mu$-integrable.
	\item\label{item:difflem2}
		For $\mu$-a.e.~$\sigma \in \Sigma$, the function $r \mapsto f(r,\sigma)$ is differentiable at $r \in I$ with derivative $f_r(r,\sigma)$.
	\item\label{item:difflem3}
		There exists a $\mu$-integrable function $g:\Sigma \to \R$ such that $|f_r(r,\sigma)| \leq g(\sigma)$ holds for all $r \in I$ and for $\mu$-a.e.~$\sigma \in \Sigma$.
\end{enumerate}
Then, $f_r(r,\cdot)$ is $\mu$-integrable for each $r \in I$, and the function $F:r \mapsto \int_\Sigma f(r,\sigma)\,\mu(d\sigma)$ is differentiable at $r \in I$ with derivative
\begin{align*}
	\frac{d}{dr} F(r)=\int_\Sigma f_r(r,\sigma)\,\mu(d\sigma).
\end{align*}
\end{lem}

The following proposition enables us to differentiate under the integral sign in the proof of Lemma~\ref{lem:derivative_b}.

\begin{prop}
Fix $y \in \Z^d \setminus \{0\}$ and define for $0<r<1$ and the trajectory $\S$ of the simple random walk,
\begin{align*}
	\phi(r,\S)
	:= \prod_{\substack{z \in \Z^d\\ \ell_z(H(y,\S),\S) \geq 1}} \E_r\bigl[ e^{-\ell_z(H(y,\S),\S)\omega(0)} \bigr] \1{\{ H(y,\S)<\infty \}}.
\end{align*}
Then, we have for all $0<r<1$,
\begin{align}\label{eq:diff_an}
	\frac{d}{dr} E^0[\phi(r,\S)]=E^0\biggl[ \frac{d}{dr} \phi(r,\S) \biggr].
\end{align}
\end{prop}
\begin{proof}
Fix $y \in \Z^d \setminus \{0\}$ and let $0<r_0<1$.
It suffices to prove \eqref{eq:diff_an} for all $0<r<r_0$.
To this end, set $I:=(0,r_0)$, $\Sigma:=(\Z^d)^{\N_0}$ (equipped with the probability measure $\mu:=P^0(\S \in \cdot)$) and
\begin{align*}
	f(r,\sigma):=\phi(r,\sigma),\qquad r \in I,\,\sigma \in \Sigma.
\end{align*}
Then, we can rewrite $E^0[\phi(r,\S)]$ as
\begin{align*}
	E^0[\phi(r,\S)]=\int_\Sigma f(r,\sigma)\,\mu(d\sigma).
\end{align*}
Hence, for \eqref{eq:diff_an}, let us check conditions \eqref{item:difflem1}--\eqref{item:difflem3} in Lemma~\ref{lem:difflem}.

Condition~\eqref{item:difflem1} is clearly satisfied since $0 \leq f(r,\sigma) \leq 1$ holds for each $r \in I$ and $\sigma \in \Sigma$.
Furthermore, by \eqref{eq:derivative_phi}, for any $\sigma \in \Sigma$,
\begin{align*}
	f_r(r,\sigma)
	= \phi(r,\sigma) \sum_{\substack{z \in \Z^d\\ \ell_z(H(y,\sigma),\sigma) \geq 1}}
	\frac{1-e^{-\ell_z(H(y,\sigma),\sigma)}}{r+e^{-\ell_z(H(y,\sigma),\sigma)}(1-r)},
\end{align*}
and condition~\eqref{item:difflem2} is valid.
It remains to check condition~\eqref{item:difflem3}.
To this end, let us observe that for fixed $\sigma \in \Sigma$ and $z \in \Z^d$ with $\ell_z(H(y,\sigma),\sigma) \geq 1$, the function
\begin{align*}
	h(r):=\frac{\phi(r,\sigma)}{r+e^{-\ell_z(H(y,\sigma),\sigma)}(1-r)}
\end{align*}
is increasing in $r \in (0,1)$.
A standard calculation shows that for each $z \in \Z^d$,
\begin{align*}
	\frac{d}{dr}h(r)
	&= \frac{\phi(r,\sigma)}{r+e^{-\ell_z(H(y,\sigma),\sigma)}(1-r)}\\
	&\quad \times \sum_{\substack{w \in \Z^d \setminus \{z\}\\ \ell_w(H(y,\sigma),\sigma) \geq 1}}
		\frac{1-e^{-\ell_w(H(y,\sigma),\sigma)}}{r+e^{-\ell_w(H(y,\sigma),\sigma)}(1-r)}
		\geq 0,
\end{align*}
which implies that $h(r)$ is increasing in $r \in (0,1)$.
Hence, we have for all $r \in I$ and $\sigma \in \Sigma$,
\begin{align*}
	|f_r(r,\sigma)|
	&\leq \sum_{\substack{z \in \Z^d\\ \ell_z(H(y,\sigma),\sigma) \geq 1}}
		\frac{\phi(r_0,\sigma)}{r_0+e^{-\ell_z(H(y,\sigma),\sigma)}(1-r_0)}\\
	&\leq r_0^{-1} \phi(r_0,\sigma) \times \#\{ z \in \Z^d:\ell_z(H(y,\sigma),\sigma) \geq 1 \}
		=:g(\sigma).
\end{align*}
Note that
\begin{align*}
	\int_\Sigma g(\sigma)\,d\mu
	\leq r_0^{-1} \tilde{\E}_{r_0}^{0,y}[\#\mathcal{A}(0,y)].
\end{align*}
By \eqref{eq:qLA}, one has
\begin{align*}
	\tilde{\E}_{r_0}^{0,y}[\#\mathcal{A}(0,y)] \leq \frac{1+\log(2d)}{-\log\{ e^{-1}+(1-e^{-1})r_0 \}}<\infty,
\end{align*}
and the integrability of $g(\sigma)$ follows.
Therefore, condition~\eqref{item:difflem3} is also satisfied.
\end{proof}

% Acknowledgements ==============================================
\section*{Acknowledgements}
The author thanks Masato Takei for useful discussions.
%The author would also like to express his profound gratitude to the reviewer for a very careful reading of the manuscript.
% =========================================================

% References ==============================================

% =========================================================


\begin{thebibliography}{10}

\bibitem{CerDem21_arXiv}
R.~Cerf and B.~Dembin.
\newblock The time constant for Bernoulli percolation is Lipschitz continuous
  strictly above $p_c$.
\newblock {\em arXiv:2101.11858}, 2021.

\bibitem{Dem20}
B.~Dembin.
\newblock Anchored isoperimetric profile of the infinite cluster in
  supercritical bond percolation is Lipschitz continuous.
\newblock {\em Electronic Communications in Probability}, 25:1--13, 2020.

\bibitem{Dem21}
B.~Dembin.
\newblock Regularity of the time constant for a supercritical bernoulli
  percolation.
\newblock {\em ESAIM: Probability and Statistics}, 25:109--132, 2021.

\bibitem{Flu07}
M.~Flury.
\newblock Large deviations and phase transition for random walks in random
  nonnegative potentials.
\newblock {\em Stochastic processes and their applications}, 117(5):596--612,
  2007.

\bibitem{Flu08}
M.~Flury.
\newblock Coincidence of {L}yapunov exponents for random walks in weak random
  potentials.
\newblock {\em The Annals of Probability}, 36(4):1528--1583, 2008.

\bibitem{FonNew93}
L.~Fontes and C.~M. Newman.
\newblock First passage percolation for random colorings of {$\mathbf{Z}^d$}.
\newblock {\em The Annals of Applied Probability}, pages 746--762, 1993.

\bibitem{Gri99_book}
G.~Grimmett.
\newblock {\em Percolation}, volume 321.
\newblock Springer Science \& Business Media, 1999.

\bibitem{Kle13_book}
A.~Klenke.
\newblock {\em Probability theory: a comprehensive course}.
\newblock Springer Science \& Business Media, 2013.

\bibitem{KosMouZer11}
E.~Kosygina, T.~S. Mountford, and M.~P. Zerner.
\newblock Lyapunov exponents of Green's functions for random potentials tending
  to zero.
\newblock {\em Probability theory and related fields}, 150(1-2):43--59, 2011.

\bibitem{Kub20_arXiv}
N.~Kubota.
\newblock Strict comparison for the Lyapunov exponents of the simple random
  walk in random potentials.
\newblock {\em arXiv:2010.08798}, 2020.

\bibitem{Law91_book}
G.~F. Lawler.
\newblock {\em Intersections of random walks}.
\newblock Springer Science \& Business Media, 2013.

\bibitem{Le17}
T.~T.~H. Le.
\newblock On the continuity of {L}yapunov exponents of random walk in random
  potential.
\newblock {\em Bernoulli}, 23(1):522--538, 2017.

\bibitem{Mou12}
J.-C. Mourrat.
\newblock {L}yapunov exponents, shape theorems and large deviations for the
  random walk in random potential.
\newblock {\em ALEA}, 9:165--209, 2012.

\bibitem{Sch14}
J.~Scholler.
\newblock On the time constant in a dependent first passage percolation model.
\newblock {\em ESAIM: Probability and Statistics}, 18:171--184, 2014.

\bibitem{Wan01}
W.-M. Wang.
\newblock Mean field bounds on {L}yapunov exponents in {$\mathbf{Z}^d$} at the
  critical energy.
\newblock {\em Probability theory and related fields}, 119(4):453--474, 2001.

\bibitem{Wan02}
W.-M. Wang.
\newblock Mean field upper and lower bounds on {L}yapunov exponents.
\newblock {\em American Journal of Mathematics}, 124(5):851--878, 2002.

\bibitem{WuFen09}
X.-Y. Wu and P.~Feng.
\newblock On a lower bound for the time constant of first-passage percolation.
\newblock {\em arXiv:0807.0839}, 2008.

\bibitem{Zer98a}
M.~P. Zerner.
\newblock {D}irectional decay of the {G}reen's function for a random
  nonnegative potential on {$\mathbf{Z}^d$}.
\newblock {\em Annals of Applied Probability}, pages 246--280, 1998.

\bibitem{Zyg09}
N.~Zygouras.
\newblock {L}yapounov norms for random walks in low disorder and dimension
  greater than three.
\newblock {\em Probability Theory and Related Fields}, 143(3-4):615--642, 2009.

\end{thebibliography}
\end{document}